\documentclass[10pt,reqno]{amsart}

\usepackage{amssymb}
\usepackage{graphicx}
\usepackage{amsmath}
\usepackage{enumerate}
\usepackage{amsthm}
\usepackage[latin1]{inputenc}
\usepackage[all]{xy}
\usepackage[latin1]{inputenc} 
\usepackage{graphicx} 
\usepackage{mathtools}
\newtheorem{theorem}{Theorem}[section]
\newtheorem{proposition}[theorem]{Proposition}

\newtheorem{lemma}[theorem]{Lemma}

\newtheorem{definition}[theorem]{Definition}

\def\b{\beta}
\def\a{\alpha}
\def\lam{\lambda}

\begin{document}
	\title[ Sania Asif \textsuperscript{1}. Wang Yao \textsuperscript{2}]{On the Conformal biderivations and conformal commuting maps on the current Lie Conformal superalgebras}
	\author{ Sania Asif \textsuperscript{1}, Wang Yao \textsuperscript{2}}
   \address{\textsuperscript{1}School of Mathematics and Statistics, Nanjing University of information science and technology, Nanjing, Jianngsu Province, PR China.}
	\address{\textsuperscript{2}School of Mathematics and Statistics, Nanjing University of information science and technology, Nanjing, Jianngsu Province, PR China.}
	

	\email{\textsuperscript{1}11835037@zju.edu.cn, 200036@nuist.edu.cn}
	\email{\textsuperscript{2}wangyao@nuist.edu.cn}
	
	
	\date{\today}
\thanks{ This work was supported by the Jiangsu Natural Science Foundation Project (Natural Science Foundation of Jiangsu Province),
	Relative Gorenstein cotorsion Homology Theory and Its Applications (No.BK20181406) }
	\begin{abstract}
	Let $L$ be a Lie conformal superalgebra and $A$ be an associative commutative algebra with unity. We define the current Lie conformal superalgebra by the tensor product $L \otimes A.$ We prove every conformal super-biderivation $\varphi_{\lambda }$ on $L$ is of the form of the centroid $Cent(L)$. Moreover, we show that every Lie conformal super-biderivation on $L\otimes A$  also has the same performance as $L$. We also prove that every Lie conformal linear super-commuting map $\varPsi_{\lambda }$ on $L \otimes A$ belongs to $Cent(L \otimes A)$, if the same holds for $L$ as well.
	\end{abstract}
\maketitle
\section{Introduction}\label{introduction}
In the past few decades, a remarkable study on the  derivations and generalized derivations on the Lie algebra, Lie conformal algebra and their subalgebras has done by various researchers in \cite{Chen-Ma-Ni, Hopkins,  Zheng-Zhang, Zhou-Chen-Ma-2, Zhou-Chen-Ma-3, Zhou-Chen, Zhao-Chen-Yuan}. The research on derivations and generalized derivations  of Lie algebra contribute a key role in the development of structure theory  of the Lie algebra. First time study on the generalization of Lie algebra and its subalgebra was done by Leger and Lucks in \cite{Leger-Lucks}. Not only the Lie derivations, but their generalizations such as  Lie triple derivations and biderivations have also notable importance in the Lie algebra. This reason has made these derivations, the center of attraction for the many researchers i.e., M\"{u}ller in \cite{Muller} was first to introduce the study of Lie triple derivation of Lie algebra. Moreover, this concept was further extended by many researchers with different perspectives, see  \cite{Lu, Ji-Wang, Asif-Wu, Xiao-Wei, Wu-Asif-Munir, Ji-Wang} for more details. In the same way, the biderivations of the Block Lie algebras, Schr\"{o}dinger-Virasoro Lie algebra  and  Lie conformal superalgebras are also investigated in \cite{Bresar-Zhao, Liu-Guo-Zhao, Wang-Yu, Huang}. 
\par Lie conformal superalgebra was first studied by Kac and Fattori, when they introduced the operator product expansion of chiral fields in conformal field theory \cite{Fattori-Kac}. Lie conformal superalgebras are very important to the study of   quantum field theory, integrable systems, vertex algebras and many more branches of physics and mathematics. The representation theories of finite simple Lie conformal superalgebras are studied and further classified in \cite{Boyllian-Kac-Liberati, Boyllian-Kac-Liberati-Rudakov, Kolesnikov}. Furthermore, some infinite Lie conformal superalgebras, such as Lie conformal superalgebras of Block type and loop super-Virasoro Lie conformal superalgebra were investigated in \cite{Dai-Han, Xia, Chen-Hong-Lie-Su, Zhao-Chen-Yuan2}.
\par By studying the Super-biderivations of the current Lie superalgebras  in  \cite{Zhao-Hassine-Chen, Eremita} and series of other papers on biderivations \cite{Tang-Meng-Chen, Chen, Fan-Dai, Xia-Wang-Han, Wang-Yu, Yuan-Tang}, we introduce the Lie conformal super-biderivations of the current Lie conformal superalgebra. We  also studied the super commuting maps on Lie conformal superalgberas and current Lie conformal superalgebra. More specifically, we show that, for every biderivation of Lie conformal superalgebra there exist a centroid of the form Eq \ref{aiza}. We show that the same concept holds for the current Lie conformal algebra (i.e. tensor product of Lie conformal algebra with associative commutative algebra with identity $L\otimes A$). We proved it with the help of various lemmas.
Moreover, after defining  super commuting map on $L$ and $L\otimes A$, we show the relationship of linear super commuting map of $L\otimes A$ with the centroid of $L\otimes A$. We are optimistic that this work would contribute in the structural theory development of Lie conformal algebra and further it  would serve to many branches of mathematics and physics.
\par The organization of paper is given as follows. In Section $2$, we recollect some  fundamental definitions and show the relationship of biderivations of Lie conformal superalgebra  $L$ with the centroid of $L$. In Section $3$, we
show that if every skew-symmetric super-biderivation on  Lie conformal superalgebra $L$ is of the form of Eq \ref{aiza}, then same holds true for the current Lie conformal algebra $L \otimes A$. In Section $4$, we prove that if each linear super-commuting map on $L$ belongs to centroid of $L$, then linear super-commuting map on $L \otimes A$ also belongs to $Cent(L \otimes A)$.
\par  For convenience we call skew-symmetric super-biderivation by super-biderivation.
\section{preliminaries}
A $Z_2$-graded algebra  is a vector space $V$, such that $V =
V_0 \oplus V_1$. An element $x \in V_i$ is called $Z_2$-homogeneous element  with the degree $|x|$.
 \begin{definition}\label{11}A Lie conformal superalgebra $L$ is a $Z_{2}$-graded $\mathbb{C}[\partial]$-module equipped with the $\lambda$-bracket , which is a $\mathbb{C}$-linear map from $L\otimes L \to \mathbb{C}[\lambda]\otimes L$, $x\otimes y \mapsto [x_\lambda y]$ and  , satisfies the following properties:
		\begin{align}
		[\partial x_\lambda y]&=- \lambda[x_\lambda y], \quad  [x_\lambda \partial y]= (\partial+\lambda)[x_\lambda y] \quad (conformal \  sesquilinearity),\\
		{}[x_\lambda y]&= -(-1)^{|x||y|}[y_{-\lambda- \partial} x]\quad (skew\text{-}symmetry),\\
		{}[x_\lambda [y_\mu z]]&= [[x_\lambda y]_{\lambda+ \mu} z]+ (-1)^{|x||y|}[ y_\mu [x_\lambda z]]\quad (Jacobi\  identity),
		\end{align}
		for $x, y, z \in L$.	
\end{definition}
\begin{definition}\label{12}
	Let $L$ be a Lie conformal superalgebra. A conformal super-biderivation of $L$ is a conformal bilinear map $\varphi_{\lambda}: L \times L \to C[\lambda]\otimes L$, such that it  satisfies the following  set of equations:
\begin{equation}
\varphi_\lambda(x, y) = -(- 1)^{|x||y|}\varphi_{-\partial- \lambda}(y, x),
\end{equation}	
\begin{equation}
\varphi_\lambda (x,[y_\mu z])= [\varphi_\lambda(x, y)_{\lambda+\mu}z]  +(- 1)^{|x||y|} [y_\mu \varphi_\lambda (x, z)]
\end{equation}
\begin{equation}
\varphi_{\lambda+\mu}([x_\mu y], z)=[x_\mu \varphi_{\lambda}(y,z)]-(-1)^{|x||y|} [y_{\lambda}\varphi_{\mu}(x, z)] ,\end{equation} 
\end{definition}Above two equations are equivalent to each other.
A super-biderivation $\varphi_{\lambda}$  of $L$ of the degree $|\varphi|\in Z_{2}$ is a super-biderivation such that $\varphi_\lambda([{L_{\a}}_\mu L_\b]) \subseteq L_{\gamma+ \b+ \a}$ for any $\alpha, \beta, \gamma \in Z_2$. We denote  the set of all skew-symmetric conformal super-biderivations by $BDer_{\gamma}(L)$ with degree $\gamma$. Moreover, $$BDer(L)=  BDer_{\bar{1}}( L)+ BDer_{\bar{0}}(L).$$
\begin{definition}
	A centroid on $Z_2$ graded Lie conformal superalgebra $L = L_0 + L_1$ is  a linear map such that $\alpha_\lambda : L \to C[\lambda]\otimes L$ a linear map. Then
	$Cent (L) =\{\alpha_\lambda : L \to C[\lambda]\otimes L :\alpha_\lambda([x_{\mu} y])=  (-1)^{|x||\a|}[x_{\mu} \alpha_{\lambda}(y) ] 
	~~\forall~~ x, y \in L\}$
	is called the centroid on $L$. The set of all  centroid of degree $|\a|$ is denoted by $Cent_{\a}(L) $. We can write, $Cent(L) = Cent_1(L)
	 +Cent_0 (L) $
	\begin{equation}
	\begin{aligned}
	\alpha_{\lambda}[x_{\mu} y]=&-(-1)^{|x||y|}\alpha_\lambda[y_{-\partial-\mu} x]\\= & -(-1)^{|x||y|}(-1)^{|y||\a|}[y_{-\partial-\mu} \alpha_{\lambda}(x)]
	\\=&(-1)^{|x||y|}(-1)^{|y||\a|}(-1)^{(|x|+|\a|)|y|}[\alpha_{\lambda}(x)_{\mu} y]
	\\=&[\alpha_\lambda(x)_{\mu} y]
	\end{aligned}
	\end{equation}
So we can say that  $Cent (L) =\{\alpha_{\lambda} : L \to C[\lambda]\otimes L :\alpha_\lambda([x_\mu y])=  (-1)^{|x||\a|}[x_{\mu} \alpha_{\lambda}(y) ] =[\alpha_\lambda(x)_{\mu} y],~~
	\forall~~ x, y \in L\}$.
\end{definition}
\begin{definition} Let  $S$ be a subset of  a Lie conformal superalgebra $L$, then  centralizer of $S$ in $L$ is given by
\begin{equation}Z_L(S) = \{x \in L : [x_\lambda y] = 0 ,~~~~\forall y \in S\}.
\end{equation}
\end{definition}
\begin{definition}Center of  a Lie conformal superalgebra $L$ is defined by
	\begin{equation}
	Z_L(L)= Z(L) = \{x \in L :[x_\lambda y] = 0,~~ \forall y \in L\}.
	\end{equation}
\end{definition}
\begin{lemma}\label{s}
Let a conformal super-biderivation of $L$ is denoted by $\varphi_{\lambda }$,  then we have
\begin{equation}\label{n}[(\varphi_\lambda(x, y))_{\lambda+ \mu}[w_\gamma v]] = [[x_\mu y]_{\mu+ \gamma} \varphi_{\lambda}(w, v)],
\end{equation} for all homogeneous $x, y,  v, w \in L$
\end{lemma}
\begin{lemma}\label{a}
	Let $\varphi_{\lambda }$ be a super-biderivation  of  a Lie conformal superalgebra $L$, then we have 
	\begin{equation}[\varphi_\lam(x, y_1)_{\lam+\mu} y_2]- \varphi_{\lambda+\mu}([x_\mu y_1], y_2) \in Z_{L}(L)	\end{equation}for all $x,y_1,y_2 \in L.$
\end{lemma}
\begin{proof}
	As we know by Lemma \ref{s} that, \begin{equation}[[x_\mu y]_{\mu+\gamma} \varphi_{\lambda}(w, v)] = [(\varphi_\lambda(x, y))_{\lambda+\mu}[w_\gamma v]].
	\end{equation}
	 By replacing $y$ by $[{y_1}_\eta y_2]$ to both side of above equation, we have
	\begin{equation}\label{i}[[x_\mu [{y_1}_\eta y_2]]_{\mu+ \gamma} \varphi_{\lambda}(w, v)] = [(\varphi_\lambda(x, [{y_1}_\eta y_2]))_{\lambda+ \mu}[w_\gamma v]].
	\end{equation} By Jacobi identity left side become
	\begin{equation*}[([[x_\mu {y_1}]_{\eta+ \mu} y_2]+(-1)^{xy_1}[{y_1}_\eta[x_\mu y_2]])_{\mu+\gamma} \varphi_{\lambda}(w, v)] = [(\varphi_\lambda(x, [{y_1}_\eta y_2]))_{\lambda+\mu}[w_\gamma v]],
	\end{equation*}
	\begin{equation}[([[x_\mu {y_1}]_{\eta+\mu} y_2])_{\mu+\gamma} \varphi_{\lambda}(w, v)]+(-1)^{xy_1}[([{y_1}_\eta[x_\mu y_2]])_{\mu+\gamma} \varphi_{\lambda}(w, v)] = [(\varphi_\lambda(x, [{y_1}_\eta y_2]))_{\lambda+\mu}[w_\gamma v]],
	\end{equation}
	By Eq \ref{n}, we have\begin{equation}\label{t}[\varphi_\lambda([x_\mu y_1], y_2)_{\lambda+\mu}[w_\gamma v ]]+ (-1)^{xy_1}[\varphi_\lambda(y_1, [ x_\mu y_2])_{\lambda+\mu}[w_\gamma v ]] = [(\varphi_\lambda(x, [{y_1}_\eta y_2]))_{\lambda+\mu}[w_\gamma v]].
	\end{equation}	
    Comparing Eqs \ref{i} and \ref{t}
	\begin{equation}[[x_\mu [{y_1}_\eta y_2]]_{\mu+ \gamma} \varphi_{\lambda}(w, v)] =[\varphi_\lambda([x_\mu y_1], y_2)_{\lambda+ \mu}[w_\gamma v ]]+ (-1)^{xy_1}[\varphi_\lambda(y_1, [ x_\mu y_2])_{\lambda+\mu}[w_\gamma v ]],\end{equation}
	or\begin{equation*}[(\varphi_\lambda(x, [{y_1}_\eta y_2]))_{\lambda+ \mu}[w_\gamma v]]= [\varphi_\lambda([x_\mu y_1], y_2)_{\lambda+ \mu}[w_\gamma v ]]+ (-1)^{xy_1}[\varphi_\lambda(y_1, [ x_\mu y_2])_{\lambda+\mu}[w_\gamma v ]]\end{equation*}
	\begin{equation}0= [(-\varphi_{\lambda}(x, [{y_1}_\eta y_2])+ \varphi_\lambda([x_\mu y_1], y_2)+ (-1)^{xy_1}\varphi_\lambda(y_1, [ x_\mu y_2]))_{\lambda+ \mu}[w_\gamma v]].\end{equation}
	Now we can write
	\begin{equation}0= [(\varphi_{\lambda}(x, [{y_1}_\mu y_2])+ (-1)^{y_2(x+ y_1)}\varphi_{-\partial - \lambda}(y_2, [x_\mu y_1], )- (-1)^{xy_1}\varphi_\lambda(y_1, [ x_\mu y_2]))_{\lambda+ \mu}[w_\gamma v]].\end{equation}
	As we now that $\varphi_\lam$ is a skew symmetric bi-derivation, then it follows from the above equation that
	\begin{equation}
	[(2[\varphi_\lam(x, y_1)_{\lam+ \mu} y_2]-2[x_{\mu}\varphi_\lambda(y_1, y_2)]+ 2(-1)^{xy_1}[{y_1}_\mu\varphi_{\lambda }(x, y_2)])_{\lambda+ \mu}[w_\gamma v]]=0.
	\end{equation}
	From this we can easily get
	\begin{equation}[\varphi_\lam(x, y_1)_{\lam+\mu} y_2]- \varphi_{\lambda+\mu}([x_\mu y_1], y_2) \in Z_{L}(L),
	\end{equation}
	as desired.
    \end{proof}
    \begin{lemma}\label{sania}	Let $\varphi_{\lambda }$ be a super-biderivation  of  a perfect Lie conformal superalgebra $L$, then we have
	\begin{equation}\label{asif}[\varphi_\lam(x, y_1)_{\lam+ \mu} y_2]= \varphi_{\lambda+ \mu}([x_\mu y_1], y_2)
	\end{equation}
\end{lemma}
\begin{proof}As the center of perfect Lie conformal superalgebra $L$ is zero, so with the help of Lemma \ref{a}, we have	\begin{equation}[\varphi_\lam(x, y_1)_{\lam+\mu} y_2]- \varphi_{\lambda+\mu}([x_\mu y_1], y_2)=0
\end{equation} Hence we get our desired result.
\end{proof}
\begin{theorem}Every skew symmetric biderivaton of a center less perfect Lie conforml superalgebra $L$ is of the form \begin{equation}
	\varphi_{\lambda}(x,y)=\a_{\mu}([x_\lam  y]), ~~~~~~ \forall ~x,y\in L.
	\end{equation} \end{theorem}
     \begin{proof}
	Let us define a  linear map $\a_{\mu}:L\to C[\lambda]\otimes L$ such that
	\begin{equation}\label{aiza}
	\a_{\mu}([x_\lam  y])=\varphi_{\lambda}(x,y),~~~~~~ \forall ~x,y\in L.
	\end{equation}
	we first show by the Lemma \ref{sania} that , $\a_{\mu}$ is well defined. For this we assume that $x=\sum_{i}[{x_i}_\mu y_i]=0$, so we have
	\begin{equation}
	\begin{aligned}
	0&=\varphi_{\lam}(\sum_{i}[{x_i}_\mu y_i], z)\\&=\sum_{i}\varphi_{\lam}([{x_i}_\mu y_i],z)\\&=\sum_{i}\varphi_{\gamma}([{x_i}_\mu y_i]_{\lambda} z) \\&=\sum_{i}([\varphi_{\mu}(x_i,y_i )_{\lam} z]
	\end{aligned}
	\end{equation} As center of $L$ is zero, we have that $\varphi_{\mu}(x_i, y_i )=0$. Moreover, we  can write Eq \ref{asif} as, $\varphi_{\lam}(x,z)= [(\varphi_{\ \mu }(x))_{\lam}z]$. So this expression along with Eq \ref{aiza} , implies that $\a_{\mu}$ is in the centroid of $L$.
\end{proof}
\section{Lie conformal super-biderivations on $L \otimes A$}
Let  a Lie conformal superalgebra is denoted by $L$ and an associative commutative algebra with unity is denoted by $A$,  then we define $L \otimes A$ is a  tensor product superalgebra over a field $F$. Under the Lie conformal bracket operation  $L \otimes A$ becomes a Lie conformal superalgebra and we name it current Lie conformal superalgebras. For more detail about it consider the lemma given below. 
\begin{lemma}Tensor product superalgebra $L \otimes A$ of a Lie  conformal superalgebra $L$ and an associative commutative algebra $A$,  under the Lie conformal bracket becomes a Lie confomal superalgebra. where the bracket is defined as follows
	\begin{equation*}
	[(x \otimes a)_\lambda (y \otimes b)]=  [x_\lambda y ]\otimes ab,~~\textit{for all}  ~~x, y \in L ~~\textit{and} ~~a, b \in A.
	\end{equation*}
\end{lemma}
\begin{proof}In order to prove $L \otimes A$ is a Lie conformal superalgbra, we need to satisfy three axioms of definition \ref{11}.\begin{itemize}
		\item Conformal sesqui-linearity:
		\item Skew symmetry: \begin{equation*}\begin{aligned}
		[( x \otimes a)_\lambda (y \otimes b) ]&= [ x_\lambda y] \otimes ab\\&= -(-1)^{|x||y|} [y_{-\partial-\lambda} x ]\otimes ba
		\\&=-(-1)^{|x||y|} [ (y \otimes b)_{-\lambda-\partial} (x \otimes a) ]
		\end{aligned}
		\end{equation*}
		for all $x, y \in L~~~\textit{and}~~~a, b \in A$.
		\item Graded Jacobi identity on $L \otimes A$\begin{equation*}
		\begin{aligned}&[(x \otimes a)_\lambda  [ (y \otimes b)_\mu (z \otimes c)] ]\\&= [(x \otimes a)_\lambda  ( [y_\mu z] \otimes bc)]  \\&=[[x _\lambda [y_\mu z]] \otimes abc]  \\&=
		([[ x_\lambda y] _\mu z]   +(-1)^{|x||y|}[y_\mu  [x_\lambda z]])\otimes abc \\&= [[ x_\lambda y] _\mu z] \otimes abc  +(-1)^{|x||y|}[y_\mu  [x_\lambda z]]\otimes abc
		\\&= [[ x_\lambda y] _\mu z] \otimes abc  +(-1)^{|x||y|}[y_\mu  [x_\lambda z]]\otimes bac
		\\&= [[ (x\otimes a)_\lambda (y\otimes b)] _\mu (z\otimes c)] +(-1)^{|x||y|}[(y\otimes b)_\mu  [(x\otimes a)_\lambda (z\otimes c)]]
		\end{aligned}\end{equation*}for all $x, y, z \in L~ and ~ a, b,c \in A$ .
	\end{itemize}
	Hence $L \otimes A $ is a Lie conformal superalgebra.
\end{proof}
\begin{lemma}\label{zari}Let a current Lie conformal superalgebra  $L \otimes A$ and let basis of $A$ be $\{b_{t}: t \in  T\}$. If $Z_{L}( L) =0$ and each super-biderivation on $L$ has a relation with centroid of $L$,  \begin{equation}\label{31}i.e,~~~~
	\varphi_{\lambda}(x, y)= \alpha_{\mu}([x_\lam y]),
	\end{equation}  Then for any
	 $\varphi_\lambda\in BDer_{\theta}(L \otimes A)$ there exist some $\a_{\mu t} \in  Cent (L \otimes A)$, where $t\in T$ such that
	 \begin{equation}
	 \varphi_{\lambda}( x \otimes a, y \otimes b) = (-1)^{|\a|(|x|+ |y|)}\sum_{t\in T}\a_{\mu t} ([ x_{\lambda} y ])\otimes b_{t}ab\end{equation}
	 for all $x, y \in L ~and~ a, b \in A$, where $|\a|= |\varphi|$.
     \end{lemma}
    \begin{proof}
	Let $\varphi_{\lambda}\in  BDer_{\theta}(L \otimes A)$  Then there exist some elements $\varphi_{\lambda t}( x \otimes a, y \otimes b)\in C[\lam] \otimes L$,
	such that
	\begin{equation}\label{2}
	\varphi_{\lambda}( x \otimes a, y \otimes b) = \sum_{t\in T}\varphi_{\lambda t}( x \otimes a, y \otimes b)\otimes b_t
	\end{equation}
	for all $x, y \in L ~and~ a, b \in A$ Therefore
	\begin{equation}
	\begin{aligned}
	&\varphi_{\lambda t}([x_\mu y]\otimes 1, (z\otimes 1))\otimes b_t\\
	=&\varphi_{\lambda }([x_\mu y]\otimes 1, (z\otimes 1))\\
	=&\varphi_{\lambda }([(x\otimes 1)_\mu (y\otimes 1)], z\otimes 1)\\
	=&[(x\otimes 1)_\mu \varphi_{\lambda- \mu }(y\otimes 1, z\otimes 1)]- (-1)^{|x||y|}[(y\otimes 1)_{\lambda- \mu }\varphi_{\mu}(x\otimes1, z\otimes 1)]\\
	=&-(-1)^{|x|(|y|+|z|+|\varphi|)}[ \varphi_{\lambda-\mu }(y\otimes 1, z\otimes 1)_{-\partial-\mu}(x\otimes\ 1)]\\
	&+(-1)^{|x||y|+|y|(|x|+|z|+|\varphi|)}[\varphi_{\mu}(x\otimes1, z\otimes 1)_{-\partial-\lambda+\mu }(y\otimes 1)]\\
	=&(-1)^{|y|(|z|+|\varphi|)}[\varphi_{\mu}(x\otimes1, z\otimes 1)_{-\partial-\lambda+\mu }(y\otimes 1)]\\
	&-(-1)^{|x|(|y|+|z|+|\varphi|)}[ \varphi_{\lambda-\mu }(y\otimes 1, z\otimes 1)_{-\partial-\mu}(x\otimes\ 1)]\\
	=&(-1)^{|y|(|z|+ |\varphi|)}[\sum_{t\in T}{\varphi_{\mu t}(x\otimes1, z\otimes 1)\otimes b_t} _{-\partial-\lambda+\mu }(y\otimes 1)]\\
	&-(-1)^{|x|(|y|+|z|+|\varphi|)}[ \sum_{t\in T}{\varphi_{(\lambda-\mu )t}(y\otimes 1, z\otimes 1)\otimes b_t} _{-\partial-\mu}(x\otimes\ 1)]\\
	=&(-1)^{|y|(|z|+|\varphi|)}[\sum_{t\in T}{\varphi_{\mu t}(x\otimes1, z\otimes 1)} _{-\partial-\lambda+ \mu }y]\otimes b_t\\
	&-(-1)^{|x|(|y|+|z|+|\varphi|)}[ \sum_{t\in T}{\varphi_{(\lambda-\mu )t}(y\otimes 1, z\otimes 1)} _{-\partial-\mu}x]\otimes b_t\end{aligned}\end{equation}It follows that,
	\begin{equation}
	\begin{aligned}
	&\varphi_{\lambda t}([x_\mu y]\otimes 1, z\otimes 1)\\&= (-1)^{|y|(|z|+ |\varphi|)}[{\varphi_{\mu t}(x\otimes1, z\otimes 1)} _{-\partial- \lambda+ \mu }y]- (-1)^{|x|(|y|+ |z|+ |\varphi|)}[ {\varphi_{(\lambda-\mu )t}(y\otimes 1, z\otimes 1)} _{-\partial-\mu}x]
	\\&= -(-1)^{|y||x|}[y_{\lambda-\mu }{\varphi_{\mu t}(x\otimes1, z\otimes 1)}]+ [ x_\mu {\varphi_{(\lambda- \mu )t}(y\otimes 1, z\otimes 1)}].\end{aligned}\end{equation}
	We can define a map $$\bar{ \varphi}_{\lambda t}: L\times L \to C[\lambda]\otimes L$$ such that $$\bar{ \varphi}_{\lambda t}(x, y)= \varphi_{\lambda }(x\otimes 1, y\otimes 1)$$ Our above equation become as
	\begin{equation}
	\bar{\varphi}_{\lambda t}([x_\mu y] , z)=  [ x_\mu {\bar{\varphi}_{(\lambda-\mu )t}(y, z)}]-(-1)^{|y||x|}[y_{\lambda-\mu }{\bar{\varphi}_{\mu t}(x, z)}].\end{equation}
	Hence we show that if $\varphi_{\lambda }\in BDer(L\otimes A)$ then $\bar{\varphi}_{\lambda t }\in BDer(L)$ for all $t\in T$. As we know that For every biderivation of Lie conformal superalgebra, there exists a centroid map that satisfies Eq \ref{31}. Now to show the relation of centroid with the biderivation of current Lie conformal superalgebra.
By use of Lemma \ref{n}, we have
\begin{equation}\label{4}
\begin{aligned}[\varphi_\lambda (x\otimes a, y\otimes b)_{\lam+\mu} [(z\otimes 1)_\gamma (w \otimes 1)]]&= [[(x\otimes a )_\lambda (y\otimes b)]_{\lambda+\mu}\varphi_\gamma(z\otimes 1, w\otimes 1)]
\end{aligned}
\end{equation} By Eqs \ref{2} and \ref{31}, right side of Eq \ref{4} is equal to 
\begin{equation}\label{5}
\begin{aligned}
&[[(x\otimes a )_\lambda (y\otimes b)]_{\lambda+\mu}\varphi_\gamma(z\otimes 1, w\otimes 1)]
\\&=(-1)^{(|z|+|w|+|\varphi|)(|x|+|y|)}[\varphi_\gamma(z\otimes 1, w\otimes 1)_{-\partial-\lambda-\mu}[(x\otimes a )_\lambda (y\otimes b)]]
\\&=(-1)^{(|z|+|w|+|\varphi|)(|x|+|y|)}[\sum_{t\in T}{\varphi_{\gamma t}(z\otimes 1, w\otimes 1)\otimes b_t}_{-\partial-\lambda-\mu}[x_\lambda y]\otimes ab]
\\&=(-1)^{(|z|+|w|+|\varphi|)(|x|+|y|)}[\sum_{t\in T}{\bar{\varphi}_{\gamma t}(z, w)\otimes b_t}_{-\partial-\lambda-\mu}[x_\lambda y]\otimes ab]
\\&=(-1)^{(|z|+|w|+|\varphi|)(|x|+|y|)}[\sum_{t\in T}{\bar{\varphi}_{\gamma t}(z, w)}_{-\partial-\lambda-\mu}[x_\lambda y]]\otimes b_tab
\\&=(-1)^{(|z|+|w|+|\varphi|)(|x|+|y|)}[\sum_{t\in T}{\a_{\mu t}([z_{\gamma} w])}_{-\partial-\lambda-\mu}[x_\lambda y]]\otimes b_tab
\\&=(-1)^{(|z|+|w|+|\varphi|)(|x|+|y|)+|\a|(|z|+|w|)}[{[z_{\gamma} w]}_{-\partial-\lambda-\mu}\sum_{t\in T}\a_{\mu t}([x_\lambda y])]\otimes b_tab
\\&=(-1)^{(|z|+|w|+|\varphi|)(|x|+|y|)+|\a|(|z|+|w|)+(|\a|+|x|+|y|)(|z|+|w|)}[\sum_{t\in T}\a_{\mu t}([x_\lambda y])_{\lambda+\mu}[z_{\gamma} w]]\otimes b_tab
\\&=(-1)^{\varphi|(|x|+|y|)}[\sum_{t\in T}\a_{\mu t}([x_\lambda y])_{\lambda+\mu}[z_{\gamma} w]]\otimes b_tab \end{aligned}
\end{equation}
Similarly by using left hand side of Eq \ref{4}, we have
\begin{equation}\label{6}
\begin{aligned}
&[\varphi_\lambda (x\otimes a, y\otimes b)_{\lam+\mu} [(z\otimes 1)_\gamma (w \otimes 1)]]\\&=[\sum_{t\in T}{\varphi_{\lambda t} (x\otimes a, y\otimes b)\otimes b_t}_{\lam+\mu} [z_\gamma w]\otimes 1]\\&=[\sum_{t\in T}{\varphi_{\lambda t} (x\otimes a, y\otimes b)}_{\lam+\mu} [z_\gamma w]]\otimes b_t.
\end{aligned}
\end{equation}
From Eqs \ref{5} and \ref{6}, we have
\begin{equation*}
(-1)^{|\varphi|(|x|+|y|)}[\a_{\mu t}([x_\lambda y])_{\lambda+\mu}[z_{\gamma} w]]\otimes b_tab-[{\varphi_{\lambda t} (x\otimes a, y\otimes b)}_{\lam+\mu} [z_\gamma w]]\otimes b_t=0
\end{equation*}
We define a map $\delta_w : A \to F$ for all $w \in T$ Obviously, for $a \in A$ there exist some elements $\delta_w(a)\in F$ such that $a = \sum_{w\in T}
\delta_w(a) b_w$, where $ w \in T$ It follows that we have
\begin{equation}
\begin{aligned}
&(-1)^{|\varphi|(|x|+|y|)}[\sum_{t\in T}\a_{\mu t}([x_\lambda y])_{\lambda+\mu}[z_\gamma w]]\otimes b_tab-\sum_{w\in T}[{\varphi_{\lambda w} (x\otimes a, y\otimes b)}_{\lam+\mu} [z_\gamma w]]\otimes b_w
\\=&(-1)^{|\varphi|(|x|+|y|)}[\sum_{t\in T}\a_{\mu t}([x_\lambda y])_{\lambda+\mu}[z_{\gamma} w]]\otimes \sum_{w\in T} \delta_w(b_tab)b_{w}-\sum_{w\in T}[{\varphi_{\lambda w} (x\otimes a, y\otimes b)}_{\lam+\mu} [z_\gamma w]]\otimes b_w
\\=&(-1)^{|\varphi|(|x|+|y|)}\sum_{w\in T}\delta_w(b_tab)[\sum_{t\in T}\a_{\mu t}([x_\lambda y])_{\lambda+\mu}[z_{\gamma} w]]\otimes b_{w}-\sum_{w\in T}[{\varphi_{\lambda w} (x\otimes a, y\otimes b)}_{\lam+\mu} [z_\gamma w]]\otimes b_w
\\=&(-1)^{|\varphi|(|x|+|y|)}\sum_{w\in T}(\sum_{t\in T}\delta_w(b_tab)[\a_{\mu t}([x_\lambda y])_{\lambda+\mu}[z_{\gamma} w]]-[{\varphi_{\lambda w} (x\otimes a, y\otimes b)}_{\lam+\mu} [z_\gamma w]])\otimes b_w	
\\=&0
\end{aligned}
\end{equation}So we have 
\begin{equation}(-1)^{|\varphi|(|x|+|y|)}\sum_{t\in T}\delta_w(b_tab)[\a_{\mu t}([x_\lambda y])_{\lambda+\mu}[z_{\gamma} w]]= [{\varphi_{\lambda w} (x\otimes a, y\otimes b)}_{\lam+\mu} [z_\gamma w]]\end{equation}
\begin{equation}\sum_{t\in T}[(-1)^{|\varphi|(|x|+|y|)}(\delta_w(b_tab)\a_{\mu t}[x_\lambda y]- {\varphi_{\lambda w }(x\otimes a, y\otimes b)})_{\lam+\mu} [z_\gamma w]]=0\end{equation}
As $Z_L(L)=0$ so we have \begin{equation}\label{7}(-1)^{|\varphi|(|x|+|y|)}\sum_{t\in T}(\delta_w(b_tab)\a_{\mu t}[x_\lambda y])= {\varphi_{\lambda w }(x\otimes a, y\otimes b)}\end{equation} and finally by using Eqs \ref{2} and \ref{7} we get
\begin{equation}\begin{aligned}
\varphi_\lambda (x\otimes a, y\otimes b)=&(-1)^{|\varphi|(|x|+|y|)}\sum_{w\in T}\varphi_{\lambda w }(x\otimes a, y\otimes b)\otimes b_{w}
\\=&(-1)^{|\varphi|(|x|+|y|)}\sum_{w\in T}(\delta_w(b_tab)\a_{\mu t}[x_\lambda y])\otimes b_w
\\=&(-1)^{|\varphi|(|x|+|y|)}\sum_{w\in T}(\a_{\mu t}[x_\lambda y])\otimes \delta_w(b_tab)b_w
\\=&(-1)^{|\varphi|(|x|+|y|)}\sum_{w\in T}\a_{\mu t}([x_\lambda y])\otimes b_tab \end{aligned}\end{equation}
for all $x, y \in L$ and $a, b \in A$.
\end{proof}
\begin{theorem}
Let a current Lie conformal superalgebra $L \otimes A$, let $Z_L (L)=0$ and $dimA\le \infty $. If each super-biderivation on $L$ can e writte in the form  the form of Eq \ref{aiza}, Then it also holds true for $L \otimes A$ .
\end{theorem}
\begin{proof}Since $dim(A)\le \infty$. Let $b_t$ for $t=\{1,2,...,n\}$ be the basis of $A$ and let $\varphi_{\lambda}\in BDer(L\otimes A)$. By the Lemma \ref{zari}, there exist a family of centroids 
$\alpha_{\mu t}$ of Lie conformal superalgrbra $L$, such that \begin{equation}
\varphi_\lambda(x\otimes a, y\otimes b)= (-1)^{|\varphi|(|x|+|y|)}\sum_{w\in T}\a_{\mu t}([x_\lambda y])\otimes b_tab
\end{equation}
Now	we  define a map $\epsilon_t: L\otimes A\to L\otimes A$ such that $\epsilon_t= \a_t+\b_{b_t}$, where $\b_{b_t}:A\to A$ defined by $a\mapsto b_{t}a$. It is clear that $\epsilon_t\in Cent(L\otimes A)$, Also assume that $\epsilon=\sum_{t\in T}\epsilon_t$, then $\epsilon$ is also in $Cent(L\otimes A)$. further we have \begin{equation}\begin{aligned}
\varphi_\lambda(x\otimes a, y\otimes b)
&= (-1)^{|\varphi|(|x|+|y|)}\sum_{w\in T}\a_{\mu t}([x_{\lambda t} y])\otimes b_tab
\\&=(-1)^{|\varphi|(|x|+|y|)} \sum_{t\in T}(\a_{\mu t}+ \b_{b_t})([x_{\lam t} y]\otimes ab)
\\&=(-1)^{|\varphi|(|x|+|y|)}\sum_{t\in T}\epsilon_{\mu t }([x_{\lam t} y]\otimes ab)
\\&=(-1)^{|\varphi|(|x|+|y|)}\sum_{t\in T}\epsilon_\mu ([x_\lam y]\otimes ab)
\end{aligned}
\end{equation}
\end{proof}
\section{Lie Conformal super-commuting maps on $L \otimes A$}
\begin{definition}\label{41}
A map $\varPsi_\lambda : L \to C[\lambda]\otimes L $  is called super-commuting map of  Lie conformal superalgebra $L$, if $Z_2$-grading of $L$ is preserves and $[\varPsi_\lambda (u)_{\lambda+\mu} u] =0$, for all $u \in L$.
\end{definition}
\begin{proposition}\label{42}
$\varPsi$ is a linear super-commuting map on $L$, then
\begin{equation*}
[\varPsi_\lambda (u)_{\lambda+\mu} v] = (-1)^{|u|(|\varPsi|+|v|)}[u_{-\partial-\lambda-\mu} \varPsi (v)]
\end{equation*}for all $u, v\in L$.
\end{proposition}
\begin{proof}
Let $\varPsi_\lam$ be a linear super-commuting map on $L$. Then\begin{equation*}
\begin{aligned}0&=[\varPsi_\lambda(u+v)_{\lam+\mu} (u+v)]
\\&=[\varPsi_{\lam}(u)+ \varPsi_\lambda(v)_{\lambda+\mu} u+ v]
\\&=[\varPsi_\lambda(u)_{\lambda+\mu} u]+ [\varPsi_\lambda(v)_{\lambda+\mu} u]+[\varPsi_\lambda(u)_{\lambda+\mu} v]+ [\varPsi_\lambda(v)_{\lambda+\mu} v]
\\&=[\varPsi_\lambda(v)_{\lambda+\mu} u]+[\varPsi_\lambda(u)_{\lambda+\mu} v]
\\&=[\varPsi_\lambda(u)_{\lambda+\mu} v]-(-1)^{|u|(|\varPsi|+|v|)}[u_{-\partial-\lambda-\mu} \varPsi_{\lambda}(v)]
\end{aligned}
\end{equation*}
for all $u, v \in L$  we have
\begin{equation*}
0=[\varPsi_\lambda(u)_{\lambda+\mu} v]-(-1)^{|u|(|\varPsi|+|v|)}[u_{-\partial-\lambda-\mu} \varPsi_\lam(v)]
\end{equation*}for all $u, v \in L$.
\end{proof} 
\begin{proposition}
Let	$\varPsi_\lam$ is a linear super-commuting map on $L\otimes A$, then
\begin{equation*}
[\varPsi_\lambda (u\otimes a)_{\lambda+\mu} (v\otimes b)] = (-1)^{|u|(|\varPsi|+|v|)}[(u\otimes a)_{-\partial-\lambda-\mu} \varPsi (v\otimes b)]
\end{equation*}for all $u, v\in L$.
\end{proposition}
\begin{proof}The proof to the proposition is straight forward and direct consequence of Proposition \ref{42}.
\end{proof}
\begin{theorem}Let  a current Lie conformal superalgebra $L \otimes A$, where $Z_L(L')= 0$. If every linear super-commuting map on $L$ is of the form of $Cent(L)$. Then it also holds true for $L \otimes A$
\end{theorem}
\begin{proof}
Let  a linear super-commuting map  of the current Lie conformal superalgebra $(L \otimes A)$ is $\varPsi_{\lambda}$. Then there exist some elements $\varPsi_{\lambda t}( u\otimes a)\in L$, such that for $t\in T$.
\begin{equation}\label{43}
\varPsi_{\lambda}(u\otimes a) = \sum_{t\in T}\varPsi_{\lambda t}(u\otimes a)\otimes b_t
\end{equation} for all $u, v\in L ~and~ a, b \in A$. Therefore
\begin{equation*}
\begin{aligned}
0=&[\varPsi_{\lambda }(u\otimes 1)_{\lam +\mu} (u\otimes 1)]
\\=&[{\varPsi_{\lambda t}(u\otimes 1)\otimes b_{t}}_{\lam +\mu} (u\otimes 1)]
\\=&[{\varPsi_{\lambda t}(u\otimes 1)}_{\lam +\mu} u]\otimes b_{t}\end{aligned}\end{equation*}
As $Z_{L}(L)=0$ , so we have
\begin{equation*}
[\varPsi_{\lambda t}(u\otimes 1)_{\lam +\mu} u]=0.
\end{equation*}
We define a map $$\bar{ \varPsi}_{\lambda t}: L \to L[\lambda]\otimes A$$ such that $$\bar{ \varPsi}_{\lambda t}(u)= \varPsi_{\lambda }(u\otimes 1)$$ So obviously, $ \bar{ \varPsi}_{\lambda t}$ is a linear super-commuting map on $L$ and $|\varPsi|=|\bar{ \varPsi}|$. Moreover, by Proposition \ref{42} and Eq \ref{43}, we have 
\begin{equation*}
\begin{aligned}0=&[\varPsi_{\lambda} (u \otimes a )_{\lam+\mu} (v \otimes 1)] -(-1)^{|u|(|\varPsi|+|v|)} [(u \otimes a)_{-\partial- \lambda- \mu}\varPsi_{\lambda }(v \otimes 1)]
\\=& [\varPsi_{\lambda} (u \otimes a )_{\lam+ \mu} (v \otimes 1)] +(-1)^{|u|(|\varPsi|+|v|)+((|\varPsi|+|v|)|u|)} [\varPsi_{\lambda }(v \otimes 1)_{\lambda+\mu}(u \otimes a)]
\\=& [\varPsi_{\lambda} (u \otimes a )_{\lam+ \mu} (v \otimes 1)]+[\varPsi_{\lambda }(v \otimes 1)_{\lambda+\mu}(u \otimes a)]
\\=& [\sum_{t\in T}{\varPsi_{\lambda t} (u \otimes a )\otimes b_t}_{\lam+ \mu} (v \otimes 1)]+[\sum_{t\in T}{\varPsi_{\lambda t}(v \otimes 1)\otimes b_t}_{\lambda+\mu}(u \otimes a)]
\\=& \sum_{t\in T}[{\varPsi_{\lambda t} (u \otimes a )}_{\lam+ \mu} v ]\otimes b_t+\sum_{t\in T} [{\varPsi_{\lambda t}(v \otimes 1)}_{\lambda+\mu}u]\otimes b_ta
\\=& \sum_{t\in T}[{\varPsi_{\lambda t} (u\otimes a)}_{\lam+ \mu} v]\otimes b_t+ \sum_{t\in T}[{\bar{\varPsi}_{\lambda t}(v)}_{\lambda+\mu}u]\otimes b_ta 
\end{aligned}
\end{equation*}
We define a map $\delta_w : A \to F$ for all $w \in T$ Obviously, for $a \in A$ there exist some elements $\delta_w(a)\in F$ such that $a = \sum_{w\in T}
\delta_w(a) b_w$, where $ w \in T$. It follows that we have
\begin{equation*}
\begin{aligned}
&\sum_{t\in T}[{\varPsi_{\lambda t} (u \otimes a)}_{\lam+ \mu} v]\otimes b_t+ \sum_{t\in T} [{\bar{\varPsi}_{\lambda t}(v)}_{\lambda+\mu}u ]\otimes b_ta
\\&=\sum_{t\in T}[{\varPsi_{\lambda w} (u \otimes a)}_{\lam+ \mu} v]\otimes b_w+ \sum_{t\in T} [{\bar{\varPsi}_{\lambda t}(v)}_{\lambda+\mu}u ]\otimes (\sum_{w\in T}\delta_w(b_ta)b_w)
\\&=\sum_{w\in T}[{\varPsi_{\lambda w} (u \otimes a)}_{\lam+ \mu} v]\otimes b_w+ \sum_{t\in T}\sum_{w\in T}\delta_w(b_ta)[{\bar{\varPsi}_{\lambda t}(v)}_{\lambda+\mu}u ]\otimes b_w
\\&=\sum_{w\in T}[{\varPsi_{\lambda w} (u \otimes a)}_{\lam+ \mu} v]\otimes b_w+\sum_{t\in T} \sum_{w\in T}\delta_w(b_ta) (-1)^{|u|(|\bar\varPsi|+ |v|)}[v_{-\partial- \lambda-\mu} {\bar{\varPsi}_{\lambda w}(u)}]\otimes b_w
\\&=\sum_{w\in T}[{\varPsi_{\lambda w} (u \otimes a)}_{\lam+ \mu} v]\otimes b_w-(-1)^{|\bar{\varPsi}|(|u|+ |v|)} \sum_{t\in T}\sum_{w\in T}\delta_w(b_ta)[{\bar{\varPsi}_{\lambda t}(u)}_{\lambda+ \mu} v]\otimes b_w.\end{aligned}
\end{equation*}
Since $Z_{L}(L)=0$, we have
\begin{equation*}\varPsi_{\lambda w} (u \otimes a)= (-1)^{|\bar\varPsi|(|u|+ |v|)} \sum_{t\in T}\delta_w(b_ta)\bar{\varPsi}_{\lambda t}(u)
\end{equation*}
for all $u \in L$ and $a\in A$. It follows that
\begin{equation*}\begin{aligned}
\varPsi_{\lambda } (u \otimes a)=&\sum_{w\in T}\varPsi_{\lambda w} (u \otimes a)\otimes b_{w}
\\=& (-1)^{|\bar\varPsi|(|u|+ |v|)}\sum_{w\in T}\sum_{t\in T}\delta_w(b_ta)\bar{\varPsi}_{\lambda t}(u)\otimes b_w
\\=& (-1)^{|\bar\varPsi|(|u|+ |v|)}\sum_{t\in T}\bar{\varPsi}_{\lambda t}(u)\otimes \sum_{w\in T}\delta_w(b_ta) b_w
\\=& (-1)^{|\bar\varPsi|(|u|+ |v|)}\sum_{t\in T}\bar{\varPsi}_{\lambda t}(u)\otimes b_ta
\end{aligned}
\end{equation*}
for all $u \in L$ and $a \in A.$ By Proposition \ref{42}, we have
\begin{equation*}
\begin{aligned}
\varPsi_\lambda[(u\otimes a)_{\mu} (v\otimes b)]=&\varPsi_{\lambda}([u_{\mu} v]\otimes ab)
\\=&(-1)^{|\bar\varPsi|(|u|+|v|)}\sum_{t\in T}\bar{\varPsi}_{\lambda t}([u_{\mu} v])\otimes b_t (ab)
\\=&(-1)^{|\bar\varPsi|(|u|+|v|)}(-1)^{|\bar\varPsi||u|}\sum_{t\in T}([u_{\mu} \bar{\varPsi}_{\lambda t}(v)])\otimes b_t (ab)
\\=&(-1)^{|\bar\varPsi||u|}([(u\otimes a)_{\mu} (\sum_{t\in T}(-1)^{|\bar\varPsi|(|u|+|v|)}\bar{\varPsi}_{\lambda t}(v)\otimes b_t (b))])
\\=&(-1)^{|\varPsi||u|}[(u\otimes a)_{\mu} {\varPsi}_\lambda(v\otimes b)]
\end{aligned}
\end{equation*}
for all $u, v \in L$ and $a, b \in A$ Hence $\varPsi_\lambda \in Cent (L\otimes A)$.
\end{proof}

\end{document}